\documentclass{emsprocart}

\contact[Antti.Kupiainen@helsinki.fi]{Antti Kupiainen, Department of Mathematics and Statistics, University of Helsinki, PO Box 68, 00014 Helsinki, Finland}

\newtheorem{theorem}{Theorem}[section]
\newtheorem{corollary}[theorem]{Corollary}

\newtheorem{proposition}[theorem]{Proposition}

\newtheorem*{coro}{Corollary} 

\theoremstyle{definition}
\newtheorem{definition}[theorem]{Definition}

\newcommand{\C}{\mathbb{C}}

\newcommand{\D}{\mathbb{D}}
\newcommand{\R}{\mathbb{R}}
\newcommand{\Z}{\mathbb{Z}}

\newcommand{\E}{\mathbb{E}}
\newcommand{\T}{\mathbb{T}}
\renewcommand{\P}{\mathbb{P}}
\newcommand{\hf}{\frac{_1}{^2}}

\newcommand{\caD}{{\mathcal D}}
\newcommand{\caE}{{\mathcal E}}
\newcommand{\caF}{{\mathcal F}}

\newcommand{\caH}{{\mathcal H}}

\newcommand{\caL}{{\mathcal L}}
\newcommand{\caM}{{\mathcal M}}

\newcommand{\caO}{{\mathcal O}}

\newcommand{\caR}{{\mathcal R}}
\newcommand{\caS}{{\mathcal S}}
\newcommand{\caT}{{\mathcal T}}

\newcommand{\caV}{{\mathcal V}}

\newcommand{\ep}{{\epsilon}}
\newcommand{\non}{{\nonumber}}
\newcommand{\Hom}{\operatorname{Hom}}
\newcommand{\Ker}{\operatorname{Ker}}

\title[Quantum Fields and Probability]{Quantum Fields and Probability}

\author[Antti Kupiainen]{Antti Kupiainen\thanks{Supported by Academy of Finland}}
\begin{document}

\begin{abstract}
I review some recent work where ideas and methods from Quantum Field Theory have proved useful in probability and vice versa. The topics discussed include the use of Renormalization Group theory in Stochastic Partial Differential Equations driven by space-time white noise and the use of the theory of Gaussian Multiplicative Chaos in the study of two dimensional Liouville Conformal Field theory.
\end{abstract}



\maketitle

\section{Quantum Fields and Random Fields}\label{intro}

Quantum Field Theory (QFT) was originally developed as a quantum theory of physical systems with infinite number of degrees of freedom. 
Perhaps the simplest example 
is obtained from the wave equation 
$$(\partial_t^2-\Delta)\varphi=0$$
 where $\varphi(t,x)$, $x\in\R^{d-1}$. In its first order form $\dot\varphi=\pi,\dot\pi=\Delta\varphi$  it is an infinite dimensional Hamiltonian system on the phase space of suitable functions $\varphi(x),\pi(x)$. In quantum theory the field $\varphi(t,x)$ becomes an operator valued distribution i.e. the smeared fields $\varphi(f)
$ with $f\in \caS(\R\times\R^{d-1})$ act as (unbounded) operators in a Hilbert space $\caH$. The physical content of this QFT is summarized by the 
the {\it Wightman functions}
\begin{align*}
 (\Omega,\varphi(f_1)\dots\varphi(f_n)\Omega)=\int W_n(z_1,\dots,z_n)\prod f(z_i)dz_i
\end{align*}
where we denoted $(t,x)$ by $z$ and $\Omega\in\caH$ is a special vector ("vacuum"). The Wightman functions  are distributions, $W_n\in S'(\R^{nd})$ and more generally, an axiomatic characterization of QFT can be given  in terms of such distributions \cite{SW}.

In this formulation there is nothing random about QFT. However, it was later realized that the Wightman functions $W_n$ are boundary values of analytic functions in complex $z_i$ and they have in particular  an analytic continuation to the {\it Euclidean} region of imaginary time, leading to  {\it Schwinger functions}
$$S_n((t_1,x_1),\dots,(t_n,x_n)):=W_n((-it_1,x_1),\dots,(-it_n,x_n)).
$$
The { Schwinger functions}  have a probabilistic interpretation as correlation functions of a {\it random field} $\phi(x)$, $x\in\R^{d}$ (modulo a regularity assumption guaranteeing solution of a moment problem):
\begin{align*}
S_n(f_1,\dots,f_n)=\E(
\phi(f_1)\dots\phi(f_n)).
\end{align*}
In our simple example of quantization of the wave equation the resulting random field is the {\it Gaussian Free Field} (GFF), a Gaussian random distribution with covariance operator given by the Green function of the Laplace operator in $\R^{d}$.

Conversely Osterwalder and Schrader \cite{OS} showed that one can also do the opposite: starting from the Schwinger functions one can  reconstruct the Wightman functions provided the former satisfy a positivity condition called 
Reflection Positivity. In this sense the subject of QFT can be viewed as a special case of that of random fields. 

In fact, this reconstruction of the QFT plays very little role  in the major physical application of reflection positive random fields namely the theory of critical phenomena. Here the physics takes place in the Euclidean region and the coordinate corresponding to the (imaginary) time above  is a spatial variable. Examples are condensed matter systems which exhibit  scale invariance  at second order phase transition points. This scale invariance is {\it universal} i.e. independent on most microscopic details and it described  by Euclidean QFT's.

Of course in QFT  one is not interested in the free field but one with interactions. In our  wave equation example interactions enter as  nonlinear terms in the equation. The resulting random fields are {\it non-gaussian} and the goal of  {\it Constructive Field Theory} is to give examples of such fields satisfying the OS axioms. In the context of critical phenomena the statistics of the fields at the phase transition point are strongly nongaussian in the interesting cases and the scaling properties of the correlation functions differ radically from that of a Gaussian theory.

The construction of such non-gaussian random fields is by no means easy. For example if we add a simple nonlinear term $ \varphi^3$ to the wave equation then for all $d$ one needs to perform a {\it renormalization}: the equation has to be regularized by smoothening it on scale $\epsilon$ and renormalized by adding terms blowing up  as $\epsilon\to 0$. 
The way to understand this and to find the renormalizations was conceived first by physicists in the 40's studying the QFT in a formal perturbation expansion in the nonlinearity and  conceptualized by K. Wilson in the 60's \cite{wilson} in the theory of Renormalization Group.

Another vast field where ideas from QFT have proved to be very useful deals with {\it noisy dynamics}. Stationary states of Markov processes with infinite dimensional state space give rise to QFT-like random fields. Examples are nonlinear Stochastic PDE's driven by space-time white noise and scaling limits of interacting particle systems. These stationary states are usually not OS-positive but nevertheless share many features of  QFT's,  in particular the problem of renormalization.

In this presentation I will discuss two examples of interaction between QFT and Probability. The first deals with Nonlinear Parabolic PDE's driven by a very rough noise. These equations require renormalization in order to be well posed and QFT ideas are very useful in understanding of  this. 

The second example is from Constructive QFT: I will discuss how ideas from the probabilistic theory of multiplicative chaos can be used to construct an interesting Conformal Field Theory, the Liouville Theory which is  conjectured to be a basic building block of two dimensional quantum gravity or in other words the scaling limit of discrete two dimensional random surfaces.

\section{Rough SPDE's}
Nonlinear parabolic PDE's driven by a space time decorrelated noise are ubiquitous in physics.
These equations are of the form
\begin{align}
 \partial_t \phi=\Delta \phi+V(\phi)+\xi 
  \label{eq: upde}
\end{align}
where $\phi(t,x)$ is defined on $\Lambda\subset\R^d$, $V(\phi)$ is a function of $\phi$ and possibly  its derivatives which can also be non-local  and $\xi$ is white noise on $\R\times \Lambda$,
formally
\begin{align*}
 \E\ \xi(t',x')\xi(t,x)=\delta(t'-t)\delta(x'-x).
  \label{eq: white}
\end{align*}
Examples are the KPZ equation with $d=1$ and
\begin{align*}
V(\phi)=(\partial_x\phi)^2\ \ \ {\rm (KPZ)}
\end{align*}
describing random deposition in surface growth and the Ginzburg-Landau model
\begin{align*}
V(\phi)=-\phi^3\ \ \  \ {\rm (GL)}
\end{align*}
describing stochastic dynamics of  spin systems.  
 
Usually in these problems one is interested in the behavior of solutions in large time and/or long distances in space. In particular one is interested in stationary states
and their scaling properties. These can be studied with regularized versions of the equations where the
noise is replaced by a mollified version that is smooth in small scales. Often one expects that the large scale
behavior is insensitive to such regularization. 

From the mathematical point of view and sometimes also from the physical
one it is of interest to inquire about the short distance properties i.e. about the well-posedness of the equations without regularizations. Then one is encountering the problem that the solutions are expected to  have
very weak regularity, they are distributions, and it is not clear how to set up the solution theory for
the nonlinear equations in distribution spaces. 

Recently this problem was addressed by Martin Hairer \cite{{hairerkpz},hairer} who set up a solution theory for a class
of such equations, including the KPZ equation and the GL  equation  in three dimensions. The latter case was also addressed by Catellier and Chouk \cite{CC} based on the theory of paracontrolled distributions developed in \cite{GIP}.  I will discuss an alternative approach \cite{AK, KM} to this problem based on Wilson's RG ideas developed originally in QFT.

\subsection{Divergences}

To see what the problem with equations \eqref{eq: upde} is consider first the linear case $F=0$. We take the spatial domain to be the unit torus $\Lambda=\T^d=(\R/\Z)^d$. The solution with initial data $\phi_0$ is then given by
$$
\phi(t,x)=(e^{t\Delta}\phi_0)(x)+\eta(t,x)
$$
with
\begin{equation}
\eta(t)=\int_0^te^{(t-s)\Delta}\xi(s)ds
\label{etadefi}
\end{equation}
and we denote the heat semigroup by $e^{t\Delta}$.
$\eta(t,x)$ is a random field with covariance 
$$
\E\, \eta(t,x)\eta(t,y)=C_t(x,y)
$$ 
where $C_t(x,y)$ is the integral kernel of the operator
$$
\int_0^te^{2t\Delta}dt=-\hf({1-e^{2t\Delta}}){\Delta}^{-1}
$$
so $\eta$ becomes the GFF  as $t\to\infty$. In particular  $C_t(x,y)$ is {singular} in short scales :
\begin{equation}
\E\eta(t,x)\eta(t,y)\asymp \frac{1}{|x-y|^{d-2}}.\label{cova}
\end{equation}


Let now $V\neq 0$ and write  \eqref{eq: upde} as an integral equation (take  $ \phi_0=0$ )
\begin{align*}
\phi(t)&=
\int_0^te^{(t-s)\Delta}(V(\phi(s))+\xi(s))ds
=\eta(t)+\int_0^te^{(t-s)\Delta}V(\phi(s))ds\non
\end{align*}
where $\eta(t,x)$ is the solution to the linear equation.  
Fix a realization of the random field $\eta(t,x)$ and try to solve this fixed point problem by Picard iteration$$
\phi(t)=
\eta(t)+\int_0^te^{(t-s)\Delta}V(\eta(s))ds+\dots.
$$
This {fails}: for the KPZ equation $V(\eta(s))=(\partial_x\eta(s,x))^2$ is not defined as $\eta$ has the regularity of Brownian motion in $x$.  For the GL equation 
$V(\eta(s))=\eta(s,x)^3$ and by \eqref{cova} this is not defined as a random filed as $\E\eta(s,x)^3\eta(s',x')^3=\infty $. 

\subsection{Superrenormalizable QFT}
Such divergencies are familiar from QFT. Indeed, in the $V=0$ case the equation has a stationary measure $ \mu_{GFF}$ (obtained as $t\to\infty$ from \eqref{cova}) which is the Gaussian Free Field corresponding to the quantization of the linear wave equation discussed in Section \ref{intro}.  Formally
the GL equation then 
 has  a {stationary measure}
 $$
 \nu(d\phi)=e^{-\frac{1}{4}\int_{\T^d}\phi(x)^4dx}
\mu_{GFF}(d\phi) .
 $$
 This is precisely the Euclidean QFT measure corresponding to the quantization of the nonlinear wave equation. Since $\phi(x)^4$ is not a well defined random variable in $d>1$ to define this measure one needs 
{renormalization}.
First we {regularize} 
$$\phi_\epsilon(x):= (\rho_\ep\ast\phi)(x),\ \ \ \rho_\ep(x)=\ep^{-d}\rho(x/\ep)$$
where $\rho$ is a smooth mollifier and then {renormalize} by adding a {\it counter term}
 $$
 U^{(\ep)}(\phi_\ep):=\frac{_1}{^4}\phi_\ep^4+\hf r_\ep\phi_\ep^2
 $$
Then
$$
\lim_{\ep\to 0}e^{-\int_{\Lambda}U^{(\ep)}(\phi_\ep(x))dx}
\mu_{GFF}(d\phi)
$$
exists provided we take
\begin{align}
r_\epsilon=\left\{
\begin{array}{ll} 
 m\log\epsilon &d=2 \\
  m_1\epsilon^{-1}+m_2\log\epsilon &d=3 \end{array}  \right. \label{counter}
\end{align}
with suitable $m,m_1,m_2$. Construction of this limit was a major accomplishment of Constructive Field Theory in the  1970's (for references see \cite{Simon,GJ}).

\subsection{Counterterms for PDE}   

Let us take the same approach to the equation   \eqref{eq: upde} by considering a {regularized} version:
\begin{align}
\partial_t \phi=\Delta \phi+V_\epsilon(\phi)+\xi_\ep\label{1}
\end{align}
with mollified noise\footnote{In the RG setup a space time mollification is actually more natural, see \cite{AK}} $\xi_\ep(t)=\rho_\ep\ast\xi(t)$ and renormalized  $V_\epsilon$ which has $\epsilon$-dependent terms added to $V$. The noise  $\xi_\ep$ is a.s. smooth
so \eqref{1} is well posed with a.s. smooth solution  $\phi_\epsilon$. Our task is to determine $V_\epsilon$ so that  $\phi_\epsilon$ converges as $\epsilon\to 0$ to some distribution $\phi$. We consider the GL and the KPZ equations and the following generalization of the KPZ equation which shows up in fluctuating hydrodynamics \cite{Spohn} and turns out to be instructive:
\begin{align*}
\partial_t \phi_\alpha=\partial_x^2 \phi_\alpha+\sum_{\beta\gamma}{ M}_\alpha^{\beta\gamma} \partial_x { \phi}_\beta\partial_x  { \phi}_\gamma+c_{\alpha,\epsilon}+\xi_\ep
\end{align*}
where the field $\phi=(\phi_1,\phi_2,\phi_3)$ takes values in $\R^3$ and the coefficients $ M_\alpha^{\beta\gamma} \in\R$  may be quite general. For this equation we take a constant counter term $$c_{\alpha,\epsilon} =a_\alpha\epsilon^{-1}+{b}_\alpha\log\epsilon$$ and for the
GL equation we take the counter term $r_\epsilon\phi$ with $r_\epsilon$ given by \eqref{counter}. Then
\begin{theorem} There exist constants $a_\alpha,b_\alpha, m,m_1,m_2$ s.t. the following holds 
{almost surely in $\xi$}: 
There exists $T>0$ s.t. 
the regularized equation has a unique solution $\phi_\ep(t,x)$ for $t\leq T$ and 
$$
\phi_\ep\to\phi\in{\caD}'([0,T]\times \T^d)
$$
where $\phi$ is independent of the cutoff function $\rho$.
\end{theorem}
\noindent We remark that in general $b_\alpha\neq 0$ but for the KPZ case it vanishes.

\subsection{Perturbative vs. Wilsonian approach}   
The fixed point problem related to \eqref{1} is
\begin{equation}
\phi(t)=\int_0^te^{(t-s)\Delta}(V_\epsilon(\phi(s))+\xi_\epsilon(s))ds.\label{fpe}
\end{equation}
For $\ep>0$ this problem has smooth solution $\phi_\ep$ at least for some time since the noise is a.s. smooth.
However, since the limit $\phi$ will be a distribution its not clear how to set this up as a
Banach fixed point problem
\vskip 2mm
The approach in \cite{GIP,CC, hairer} is to develop a
nonlinear theory of distributions 
allowing to formulate and solve the fixed point problem. This can be compared to perturbative renormalization theory in QFT.

Our approach uses  another approach to renormalization pioneered by K. Wilson in the 60's \cite{wilson}.
In Wilson's approach adapted to the SPDE one would not try to solve  equation \eqref{eq: upde}, call
it $\caE$,  directly
but rather go scale by scale starting from the scale $\ep$ and deriving {\it effective} equations $\caE_n$ 
for
larger scales $L^n\ep:=\ep_n$, $n=1,2,\dots$ where $L>1$ is arbitrary. Going from scale $\ep_n$ to $\ep_{n+1}$ is a problem
with $\caO(1)$ cutoff when transformed to dimensionless variables. This problem can be studied by a standard Banach fixed point method. 

The possible singularities of the original problem are present in the large $n$ behavior of   the corresponding effective equation. One views $n\to\caE_n$ as a dynamical system and
attempts to find an initial condition at $n=0$ i.e. modify $\caE$ so that if we fix the scale $\ep_n=\ep'$ and then let  $\ep\to 0$ (and as a consequence $n\to\infty$) the
effective equation at scale $\ep'$ has a limit. It turns out that controlling this limit for the effective equations allows one then to control the solution to the original equation \eqref{eq: upde}.

In  this approach  no new theory of distributions needed and it provides a  general method to derive counterterms for subcritical nonlinearities as well as a  general method to study {\it universality}.

\subsection{Dimensionless variables}

For simplicity of exposition we will use a regularization in time instead of  space in the fixed point problem \eqref{fpe}:
\begin{equation}
\phi(t)=\int_0^t\chi(\frac{_{t-s}}{^{\epsilon^2}})e^{(t-s)\Delta}(V_\epsilon(\phi(s))+\xi(s))ds.\label{fpe1}
\end{equation}
where $\chi$ is smooth, vanishing in a neighborhood of $0$ and $\chi(s)=1$ for $s>1$. This cutoff has the same effect as the mollification i.e. regularizing the problem in spatial scales $<\epsilon$.

It will be useful to introduce {\it dimensionless variables} in terms of which the cutoff $\epsilon=1$.
Define  a {space time scaling} operation $s_\mu$ by
$$
 (s_\mu \phi)(t,x):=\mu^{\frac{d-2}{2}} \phi(\mu^2t,\mu x).
$$
This scaling {preserves} the linear equation $
\dot\phi=\Delta\phi+\xi
$. We will now set
$$
\varphi:=s_\ep\phi.
$$
Then the KPZ and GL nonlinear terms 
$$
V_\epsilon(\phi)= \left\{
 \begin{array}{rl}
 (\partial_x\varphi)^2+c_\epsilon& \text{KPZ } \\
   \varphi^3 +r_\epsilon\varphi& \text{GL }  \end{array} \right.
$$
become
\begin{equation*}
v^{(\epsilon)}(\varphi)=\left\{
 \begin{array}{rl}
  \epsilon^{\frac{2-d}{2}}(\partial_x\varphi)^2+\epsilon^2c_\epsilon& \text{KPZ } \\
    \epsilon^{{4-d}}\varphi^3 + \epsilon^2r_\epsilon\varphi& \text{GL }  \end{array} \right.
\end{equation*}
and the fixed point problem \eqref{fpe1} becomes
\begin{equation}
\varphi(t)=\int_0^t\chi(t-s)e^{(t-s)\Delta}(v^{(\epsilon)}(\varphi(s))+\xi(s))ds:=G(v^{(\epsilon)}(\varphi)+\xi).\label{fpe2}
\end{equation}
In this dimensionless formulation the equation has  cutoff on unit scale (instead of scale $\epsilon$) and the nonlinearity is small if $d<2$ (KPZ), $d<4$ (GL). These are the   {\it subcritical} cases.
However,  $\varphi$  is now defined on $[0,\ep^{-2}T]\times(\ep^{-1}\T)^d$ i.e. we need to  control {\it arbitrary large times and volumes} as $\ep\to 0$ (we denote the noise  in \eqref{fpe2} again by  $\xi$: it equals in law the space time white noise on $[0,\ep^{-2}T]\times(\ep^{-1}\T)^d$).

\subsection{Renormalization Group}

Fix now a scale $L>1$ and solve the equation \eqref{fpe2} for  spatial scales $\in [1, L]$ ( temporal scales $\in [1,L^{2}]$).  Concretely, we insert 
$$\chi(t-s)=\chi(L^{-2}(t-s))+(1-\chi(L^{-2}(t-s))$$
 in \eqref{fpe2} so that
$$
G=G_0+G_1
$$
where  
$G_0$ involves scales $\in [1,L]$ and $G_1$   scales $\in [L,\epsilon^{-1}]$.
The problem  \eqref{fpe2}
is equivalent to 
\begin{align}
\varphi=\varphi_0+\varphi_1\label{sol1}
\end{align}
 with
\begin{align}
\varphi_0&=G_0(v^{(\epsilon)}(\varphi_0+\varphi_1)+\xi)\label{small}\\
\varphi_1&=G_1(v^{(\epsilon)}(\varphi_0+\varphi_1)+\xi)\label{large}
\end{align}
It turns out that \eqref{small} is easy to solve: it has  time  $\caO(L^2)$, noise is smooth and nonlinearity is small. The solution $\varphi_0$ is a function of $\varphi_1$:
$\varphi_0=\varphi_0(\varphi)$. 
 Inserting this 
 to large scale equation \eqref{large} get
$$
\varphi_1=G_1(v^{(\epsilon)}(\varphi_1+\varphi_0(\varphi_1))+\xi)
$$
This equation {has scales $\geq L$. } The final step consists of 
{rescaling back to scales $\geq 1$}. Define the scaling transformation by
 \begin{equation}\label{scaling}
 (s \varphi)(t,x):=L^{\frac{2-d}{2}} \varphi(L^{-2}t, L^{-1}x).
\end{equation}
and set
$$\varphi':=s^{-1}\varphi_1.
$$
 By simple change of variables we have $s^{-1}Gs=L^{2}G$ and $s\xi\stackrel{law}=L^2\xi$ which lead to  a {\it renormalized equation} for $\varphi'$:
$$
\varphi'=
G(v'(\varphi')+\xi).
$$
where
 \begin{equation}\label{reneq}
 v'(\varphi')=L^2s^{-1}v^{(\epsilon)}(s\varphi'+\varphi_0(s\varphi'))
\end{equation}
This is of the same form as the original equation except that $\varphi'(t,x)$ has 
  $t\in[0,{\epsilon'}^{-2}T]$ and $x\in ({\epsilon'}^{-1}\T)^d$  with {$\epsilon'=L\epsilon$} and the nonlinearity has changed to $v'$.
    The map 
    $$\caR: v^{(\epsilon)}\to v':=\caR v^{(\epsilon)}$$
     is the {\it Renormalization Group } map.
 Iterating this procedure we obtain a sequence of nonlinearities $\caR^nv^{(\epsilon)}$ and equations
 \begin{align}
\varphi=
G(\caR^nv^{(\epsilon)}(\varphi)+\xi).
\label{nth}
\end{align}
whose solution $\varphi$ describes solution of original PDE on scales $\geq  L^n\epsilon$. Indeed, the iteration of the equation \eqref{sol1} leads to the construction of the solution to the original equation  \eqref{fpe1} in terms of the one of \eqref{nth}.

We can now address the $\epsilon\to 0$ limit. Let us define the {\it effective equation} for scales $\geq \mu$ by
$$
v_\mu^{(\epsilon)}:=\caR^{\log(\mu/\epsilon)}v^{(\epsilon)}
$$
We try to  fix the counter terms so that for all $\mu$ the following  limit  exists:
\begin{align}
v_\mu:=\lim_{\epsilon\to 0}v_\mu^{(\epsilon)}
\label{conti}
\end{align}

\subsection{Linerization}
The RG map 
$\caR$ \eqref{reneq} is a composition of a two maps
$$\caR=\caS\circ\caT$$
where $\caS$ is the {scaling}
$$
v(\varphi)\to ( \caS v)(\varphi)=L^{2}s^{-1}v( s\varphi)
$$ 
and $\caT$ is a translation
$$
v(\varphi)\to v(\varphi+\psi)
$$
where
 $\psi$ is a random function  of $v$
solved from the short time problem
  \begin{equation}\label{psieq}
 \psi=G_0 (v(\varphi +\psi)+\xi).
 \end{equation}
Let  $\caL $ be the linearization of $\caR$: 
$\caR v=\caL v+\caO(v^2)$.
Since to first order in $v$ we have from \eqref{psieq}
$
 \psi=G_0 \xi+\caO(v)
 $
we get
 $$
( \caL v)(\varphi)=(\caS v)(\varphi+G_0 \xi)
$$
The scaling operator $\caS$ has local eigenfunctions 
$$
\caS\varphi^k=L^{\alpha_k}\varphi^k,\ \ \ \ \alpha_k=2-(k-1)\frac{_{d-2}}{^2}
$$
$$
\caS(\nabla\varphi)^k=L^{\beta_k}(\nabla\varphi)^k,\ \ \ \ \beta_k=2-\frac{_{k+1}}{^2}\ \ \ d=1
$$
The $\alpha_k>0$ ({\it relevant}) eigenfunctions expand under $\caL$ , $\alpha_k<0$  ({\it irrelevant}) ones contract.
 For GL the relevant ones are $\phi^k$, $k\leq 4-d$ and for KPZ $(\nabla\phi)^k$, $k\leq 2$. 
Iterating one obtains for GL in $d=3$
$$
\caL^n\varphi^3=L^{n}(\varphi+\eta_{L^{-n}})^3
$$
and for KPZ
$$
\caL^n(\nabla\varphi)^2=L^{\frac{n}{2}}(\nabla\varphi+\nabla\eta_{L^{-n}})^2
$$	
The random field $\eta_{L^{-n}}$  is a sum of contributions from $n$ scales and in fact it is given by the GFF \eqref{etadefi} with small scale cutoff $L^{-n}$:
\begin{equation}
\eta_{L^{-n}}(t)=\int_0^t(\chi(L^{2n}(t-s))-\chi(t-s))e^{(t-s)\Delta}\xi(s)ds
\label{etadefi}
\end{equation}
 In particular
\begin{align}
\E(\nabla\eta_{L^{-n}}(t,x))^2\sim L^n\label{blow1}
\end{align}
and
\begin{align}
\E(\eta_{L^{-n}}(t,x))^2\sim  \left\{
 \begin{array}{rl}
 \log L^n&{d=2}\\
    L^{n} &{d=3}  \end{array} \right.
\label{blow2}
\end{align}
 For KPZ in linear approximation effective equation becomes
$$
v_\mu^{\epsilon}=\mu^\hf(\nabla\varphi+\nabla\eta_{\epsilon/\mu})^2+\mu^2c_\epsilon
$$
and for GL one gets
$$
v_\mu^{\epsilon}=\mu^{4-d}(\varphi+\eta_{\epsilon/\mu})^3+\mu^2r_\epsilon\varphi
$$
Due to \eqref{blow1} and \eqref{blow2} these have no limit  as $\epsilon\to 0$.

Why did this happen? 
For  KPZ the nonlinearity $(\nabla\varphi)^2$ is relevant with exponent $\hf$ but has size $\epsilon^\hf$ which reproduces under iteration. 
   However $\caR$ produces a {\it more relevant} term, constant in $\varphi$ with exponent $\frac{_3}{^2}$ and size  $\epsilon^\hf$.
    This expands under iteration to $(\frac{_\mu}{^\epsilon})^\frac{_3}{^2}\epsilon^\hf=\caO(\epsilon^{-1})$. 
 The
  solution is obvious: fix the constant  $c_\epsilon$ so as to cancel the divergence
  $$
  c_\epsilon=\E\, (\nabla\eta_\epsilon)^2=a\epsilon^{-1}
  $$
Then the effective equation becomes
$$
v_\mu^\epsilon=\mu^\hf[(\nabla\varphi)^2+2\nabla\varphi\nabla\eta_{\epsilon/\mu}+:(\nabla\eta_{\epsilon/\mu}
)^2:]
$$
where 
$$:(\nabla\eta_{\epsilon/\mu}
)^2:=(\nabla\eta_{\epsilon/\mu}
)^2-\E(\nabla\eta_{\epsilon/\mu})
^2$$
For the GL equation $\caR$ produces a relevant linear term in $\varphi$ with exponent $2$. Taking
 $$
  r_\epsilon=\E\, \eta_\epsilon^2
  $$
the effective equation becomes
$$
v_\mu^\epsilon=\mu^{4-d}[\varphi^3+3\varphi^2\eta_{\epsilon/\mu}+3\varphi:\eta_{\epsilon/\mu}^2:+:\eta_{\epsilon/\mu}^3:]
$$
The limits
 $$\lim_{\epsilon\to 0} :(\nabla\eta_{\epsilon/\mu}(t,x) )^2:\ \ =\ \ :(\nabla\eta(t,x)  )^2:
 $$
 \begin{align}\label{wick}
 \lim_{\epsilon\to 0}:\eta_{\epsilon/\mu}(t,x) ^k:\ \ =\ \ :\eta(t,x) ^k:
 \end{align}
 are distribution valued random fields, the Wick powers of the GFF. Hence in the linear approximation to the RG the limit \eqref{conti} exists a.s. as a distribution.
 
\subsection{Outline of the proof}
Let us start with GL in $d=2$. 
Denote the result of the linear approximation by
$$
u_\mu^\epsilon=\mu^{2}:(\varphi+\eta_{\epsilon/\mu})^3:
$$
and write 
$$
v_\mu^\epsilon=u_\mu^\epsilon+w_\mu^\epsilon.
$$
Since $\caL u_\mu^\epsilon= u_{L\mu}^\epsilon$ we get
$$
w_{L\mu}^\epsilon=\caL w_\mu^\epsilon+\caO(\mu^4).
$$

In $d=2$  we expect from the scaling eigenfunction analysis that $\|\caL\|\leq CL^2$ in a suitable space.  Thus we expect
$$
\|w_{L\mu}^\epsilon\|\leq CL^2 \|w_{\mu}^\epsilon\|+C \mu^4.
$$
Suppose, inductively in the scale $\mu$ that we have shown
 \begin{align}\label{indu}
\|w_\mu^\epsilon\|\leq \mu^{2+\delta}, \ \ \ \delta>0.
 \end{align}
Then $$
\|w_{L\mu}^\epsilon\|\leq CL^2\mu^{2+\delta}+C \mu^4\leq  (L\mu)^{2+\delta}
$$
provided we take $L>\caO(1)$ and $\mu<C(L)$. Thus we can inductively prove  \eqref{indu}  for scales $\mu\leq\mu_0
$.
This becomes a proof once we work in a suitable Banach space of $v$'s. Thus normal ordering suffices to make the PDE well posed.

Now consider $d=3$:
$$
u_\mu^\epsilon=\mu:(\varphi+\eta_{\epsilon/\mu})^3:
$$
and  this time $\|\caL\|\leq CL^{5/2}$ so that
$$
\|w_{L\mu}^\epsilon\|\leq CL^{5/2}\|w_{\mu}^\epsilon\| +C \mu^2
$$
Since $5/2>2$ this is not good! The linear part expands too rapidly compared with the smallness of the nonlinear contributions for the inducive argument to work. The remedy is  to compute $
v_\mu^\epsilon$ explicitly to the second order:
$$
v_\mu^\epsilon=u_\mu^\epsilon+U_\mu^\epsilon+w_\mu^\epsilon.
$$
If we could show that the second order term satisfies the bound
$$
\|U_\mu^\epsilon\|\leq C\mu^2
$$
we would get
$$
\|w_{L\mu}^\epsilon\|\leq L^{5/2}\|w_{\mu}^\epsilon\| +C \mu^3
$$
and since 
$5/2<3$ we may proceed inductively as in $d=2$  to show
 \begin{align*}
\|w_\mu^\epsilon\|\leq \mu^{5/2+\delta}, \ \ \ \delta>0.
 \end{align*}
for $\mu\leq\mu_0$,

However, $ \|U_\mu^\epsilon\|$ {diverges} as $\log\epsilon$!
$ U_\mu^\epsilon$ is a (nonlocal) polynomial in $\varphi$ and $\eta_{\epsilon/\mu}$. We
expand $ U_\mu^\epsilon$  in Wiener chaos (i.e. Wick polynomials). The result is 
$$
U_\mu^\epsilon= b\mu^2\log(\epsilon/\mu)\varphi+\tilde U_\mu^\epsilon
$$
where $\lim_{\epsilon\to 0}\tilde U_\mu^\epsilon$ exists as a random field .
Hence we learn that we
need to add an additional mass counter term to the equation
$$
v^{(\epsilon)}=\epsilon\varphi^3+\epsilon^2(a\epsilon^{-1}+b\log\epsilon)\varphi.
$$
In the 
original PDE this means 
$$
\phi^3\to\phi^3+(a\epsilon^{-1}+b\log\epsilon)\phi
$$
Why did this happen?
The linear term is {\it relevant} in 1st order and  neutral ({\it marginal}) in 2nd order.
Marginal terms can pile up logarithmic divergences upon iteration. The counter term prevents this. We get
$$
v_\mu^{(\epsilon)}=\mu\varphi^3+\mu^2(a\mu^{-1}+b\log\mu)\varphi+\dots
$$
Note that this is  {\it small} as long as $\mu$ is. Nothing is diverging!

In {KPZ} coupling constant is $\epsilon^\hf$ and $\|\caL\|=L^{3/2}$ so we need to expand $v_\mu^{(\epsilon)}$  to {  3rd order}. 
By "miracle" 2nd and 3rd order terms have {vanishing relevant and marginal terms}. The random fields occurring in them have  $\epsilon\to 0$ limits and no new renormalizations are needed.
This is {not true} for the {multicomponent KPZ}: this is the source of the  $\log\epsilon$ constant counter term coming from  in third order where constants are marginal.

In this heuristic discussion we have assumed perturbative terms $u^\epsilon_\mu$ have the obvious bounds in powers of $\mu$.  
This can not be true since they involve the random fields $:\eta^k:$, $:(\nabla\eta)^2:$ etc.
These noise fields belong to Wiener chaos of bounded order
and their covariance is in a suitable negative Sobolev space.  
Hypercontractivity implies good moment estimates for them and a
 Borel-Cantelli argument implies that  a.s. we can find a $\mu_0>$ s.t.  $\|u^\epsilon_\mu\|$ has a good bound for $\mu<\mu_0$.
On that event  the $\caR$ is controlled by a simple application of contraction mapping in a suitable Banach space. 
The time of existence of the original SPDE is $\mu_0^2$ and it is a.s. $>0$.

Finally, let us briefly discuss the  domain and range of $v_\mu^\epsilon(\varphi)$. Recall  $v_\mu^\epsilon=v_\mu^\epsilon(t,x;\varphi)$  is a function on space time and a nonlinear functional of the field $\varphi$. Consider first its dependence of $(x,t)$. 
In the GL case the random fields in the perturbative part $v_\mu^{(\epsilon)}$  (i.e. fields such as $\eta$, $:\eta^2:$ etc) are distributions which are in $H_{loc}^{-2}$
in their time dependence and in $H_{loc}^{-4}$ in their space dependence. This leads us to let $v_\mu^\epsilon$ take values in $H^{-2,-4}_{loc}$. 

As for the $\varphi$-dependence of $v_\mu^\epsilon(t,x;\varphi)$ we need to discuss the domain, i.e. in what space should the argument $\varphi$ be taken.
Since $\varphi$ represents the large scale part of the solution %
we can take $\varphi$ smooth. Explicitly we let
 $$\varphi\in C^{2,4}([0,\mu^{-2}T]\times \mu^{-1}\T^d)
  $$
  We then  prove that 
  $$
  v_\mu^\epsilon
 :C^{2,4}\to H^{-2,-4}_{loc}
  $$
is an analytic function in a ball of radius $\mu^{-\alpha}$, $\alpha>0$.

\subsection{Subcritical equations} KPZ$_{d=1}$ and GL$_{d<4}$ are {\it subcritical}:
the dimensionless strength of nonlinearity is small in short scales. Another example is
the {\it Sine-Gordon equation}
$$
\partial_t \phi=\Delta \phi+g\sin(\sqrt\beta\phi)+\xi
$$
After normal ordering dimensionless coupling is
$$
\epsilon^{2-\frac{\beta}{8\pi}}g.
$$
This is subcritical for $\beta< 16\pi$. Here one needs to
expand solution to order $k-1$ where $(2-\frac{\beta}{8\pi})k>2$.
So $k\to\infty$ as $\beta\uparrow 16\pi$.
It is a challenge to carry this out for all $\beta<16\pi$. Hairer and Shen have
controlled the case $\beta<\frac{32\pi}{3}$ \cite{hairers}.

\section{Liouville QFT}

The QFT's discussed in the previous Section are quite simple from the renormalization group point of view: they are {\it superrenormalizable} which means that the counter terms can be found without a multi scale analysis by looking at a few orders of perturbation series (Picard iteration above). We will now discuss another QFT, the Liouville model, that can be considered superrenormalizable but which has several  interesting features and applications. Its motivation comes from random surface theory and two dimensional quantum gravity.  I will discuss work done with F. David, R. Rhodes and V. Vargas to give a rigorous construction of the Liouville model \cite{DKRV,DKRV2, krv}.

\subsection{Random Surfaces} Let $\mathcal{T}_{N}$  be the set of triangulations of the 2-sphere ${S}^2$  with $N$  faces, three of which are   marked. 
$T\in\mathcal{T}_{N}$ is a graph with topology  of ${S}^2$  and each face has three boundary edges. We will consider a two-parameter family of probability measures $\P_{{\mu_0},\gamma}$ on 
$\mathcal{T}=\cup_N\mathcal{T}_{N}$ defined by
\begin{align}
\P_{{\mu_0},\gamma}(T)=\frac{1}{Z_{{\mu_0},\gamma}}
e^{-{{\mu_0}} N}Z_{{\gamma}}(T)
\label{Pdef}
\end{align}
if $T\in\caT_N$.  $Z_{{\gamma}}(T)$ is the partition function of  a {\it critical lattice model}
on the graph $T$. Such models are defined for $\gamma\in [\sqrt{2},2]$ and some examples are
percolation for $\gamma=\sqrt{8/3}$, Ising model  $\gamma=\sqrt{3}$, discrete GFF for $\gamma=2$, uniform spanning tree $\gamma=\sqrt{2}$.  It is known that
\begin{equation}
Z_N:=\sum\limits_{T \in \mathcal{T}_{N}} Z_{\gamma}(T)=N^{1-\frac{4}{\gamma^2}}e^{\bar \mu N}(1+o(1))\label{ASY}
\end{equation}
so that $\P_{\mu_0,\gamma}$ is defined for ${\mu_0}>\bar \mu$.   $Z_{{\mu_0},\gamma}$ diverges
 as $\mu_0\to\mu $ so that the measure concentrates on large triangulations in that limit.

Each $T$ has a natural {conformal structure} where each face $f$ is equilateral with unit area. Then there is a unique conformal map  $\psi: T\to{S}^2$ s.t. centers of marked faces map to $z_1,z_2,z_3$ . Let $\nu_T$ be the image of the area measure on $T$.
Under $\P_{{\mu_0},\gamma}$,  $\nu_{T}$ becomes a random measure $\nu_{{{\mu_0}},\gamma}$ on ${S}^2$. 

Consider now a {\it scaling limit} as follows. 
Recalling that as ${\mu_0}\downarrow\bar\mu$ typical size of triangulation diverges we define
for $\mu>0$ 
$$
\rho^{(\epsilon)}_{\mu,\gamma}:=\epsilon\nu_{{\bar\mu+\epsilon\mu},\gamma}
$$
so that the image triangles have area $\epsilon$.
It is natural to conjecture that $\rho^{(\epsilon)}_{\mu,\gamma}$ converges in law as $\epsilon\to0$ to 
a random measure $\rho_{\mu,\gamma}$. 
Since $\epsilon\nu_{T}({S}^2 )=\epsilon N$ the law of $\rho^{(\epsilon)}_{\mu,\gamma}({S}^2 ) $ is given by using \eqref{ASY} 
$$
\E[  F(\rho^{(\epsilon)}_{\mu,\gamma}({S}^2 )   )  ]= \frac{1}{Z_{\epsilon}}\sum_N e^{-\mu \epsilon N}N^{1-\frac{4}{\gamma^2}}
F(\epsilon N )+o(1).
$$
Hence this law converges to $\Gamma(2-\frac{4}{\gamma^2},\mu ) $.
We will construct a measure  with this law for its mass.

\subsection{KPZ Conjecture} Let $g(z)|dz|^2$ be a smooth conformal metric on the Riemann sphere $\hat \C=\C\cup\{\infty\}$.  Kniznik, Polyakov and Zamolochicov \cite{KPZ} argued that the random measure  $\rho_{\mu,\gamma}$ is given by
\begin{align}
\rho_{\mu,\gamma}(dz)=e^{\gamma \phi_g(z)}
dz\label{rho}
\end{align}
where $ \phi_g$ is
 the {\it Liouville field}
\begin{align}\label{Liouville field}
\phi_g:=X+\frac{_Q}{^2}\ln g
\end{align}
and
$X$ is a random field whose law  is  formally given by 
\begin{align}
\E_{\gamma,\mu}\,  f(X)=Z^{-1}\int_{Map(\C\to\R)} f(X)\, e^{-S_L(X,{g})}DX.
\label{funci}
\end{align}
where $S_L$ is action functional of the {\it Liouville model}: 
\begin{align}
S_L(X,{g}):= \frac{1}{\pi}
\int_{\C}\big(\partial_z X\partial_{\bar z} X+\frac{_Q}{^4}gR_{{g}} X  +\mu e^{\gamma \phi_g  }\big)\,dz.
\label{actio}
\end{align}
 Here 
 $R_g=-4g^{-1}\partial_z \partial_{\bar z}\log g
$ is the scalar curvature and 
$Q$ is related to $\gamma$ by
$$Q=2/\gamma+\gamma/2.$$
Furthermore the heuristic integration over $X$ in \eqref{funci} is supposed to include "gauge fixing" due to the marked points $z_1,z_2,z_3$.

\subsection{GFF} Let us first keep only the quadratic term in the action functional \eqref{actio} and try to define the linear functional
$$
\langle F\rangle=\int_{Map(\C\to\R)} F(X)e^{-\frac{1}{4\pi}
\int_{\C}|\partial_zX |^2dz}
DX
$$
We may define this in terms of the Gaussian Free Field. GFF  on the full plane is defined up to constant and we fix this by considering the field $X_g$ with zero average in the metric $g$: 
$$
m_g(X_g):=\frac{1}{\int_\C  g(z)dz}\int_\C X_g(z)\, g(z)dz=0.
$$
Then we set $X=X_g+c$, $c\in\R$  and  define
$$
\langle F\rangle=\int_\R (\E\, F(X_g+c))dc:=\int F(X)d\nu_{GFF}(X).
$$
Note that 
 $\nu_{GFF}(dX) =\P (dX_g)dc$ is {\it not} 
a probability measure. This measure  is {\it independent of the chosen metric} since 
$$X_{g'}\stackrel{law}{=}X_{g}-m_{g'}(X_g)$$
 where $m_{g'}(X_g)$ is a random constant that can be absorbed to a shift in $c$.

We can now give a tentative definition of the measure in \eqref{funci} by defining
\begin{align}
\nu_{g}=e^{- \frac{1}{4\pi}\int_{\C}(QR_{g} X  + \mu e^{\gamma X_g  }\,)gdz}\nu_{GFF} .\label{tenta}
\end{align}
However, now we encounter the problem of renormalization as $ e^{\gamma X  }$ is not defined since $X_g$ is not defined point wise. Indeed 
$\E X_g(z)X_g(z')=\ln|z-z'|^{-1}+\caO(1)$ as $z-z'\to 0$.

\subsection{Multiplicative Chaos} To define $ e^{\gamma X  }$ we proceed as in Section 2 by taking a mollified  version of GFF   $X_{g,\epsilon}$. Then
$
\E e^{\gamma X_{{g},\epsilon}(z)}\propto \epsilon^{-\frac{_{\gamma^2}}{^2}}$
and we renormalize by defining thr random measure on $\C$
$$
M_{g,\gamma,\epsilon}(dz):=\epsilon^{\frac{_{\gamma^2}}{^2}}e^{\gamma (X_{{g},\epsilon}(z)+\frac{_{ Q}}{^2}\ln g(z))}
dz
$$
Then
$$
M_{g,\gamma, \epsilon}\to M_{ g,\gamma} $$ 
 weakly in probability as $\ep\to 0$. The limit is nonzero if and only if $\gamma<2$. It is  an example of {\it Gaussian multiplicative chaos} (see \cite{DRSVreview} for a review), a random
{\it multifractal} measure on $\C$ for which   a.s.  $M_{ g,\gamma}(\C)<\infty$. We may now define
\eqref{tenta} as
\begin{align}
\nu_g=e^{- \frac{1}{4\pi}(\int_{\C}QR_{g} X gdz + \mu e^{\gamma c  }M_{ g,\gamma}(\C))}\nu_{GFF} .\label{tenta}
\end{align}

\subsection{Weyl and M\"obius invariance} We saw that $X$ is metric independent under $\nu_{GFF}$. Recaling the { Liouville field} \eqref{Liouville field} we have
\begin{proposition}\label{prop}
Let $F\in L^1(\nu_g)$ and 
$g'=e^{\varphi}g$. Then
\begin{align*}
\int F(\phi_{g'})d\nu_{g'}=e^{\frac{c_L-1}{96\pi}\int |\partial \varphi|^2\,dz+ \int 2R_{{g}}\varphi \,gdz}\int F(\phi_{g})d\nu_{g}
\end{align*}
where $c_L=1+6Q^2$.
\end{proposition}
\begin{proof}(see \cite{DKRV} for details)
By  metric independence of $X$ we replace $c+X_{g'}$ by $c+X_{g}$ so that
\begin{align*}
\int F(\phi_{g'})d\nu_{g'}&=\int F(\phi_{g}+\frac{_Q}{^2} \varphi)e^{- \frac{Q}{4\pi}\int R_{g'} g'(c+X_g)dz+\mu  e^{\gamma c} \int e^{
\frac{_Q}{^2}\ln \varphi}dM_{g,\gamma})}d\nu_{GFF} .
\end{align*}
Use $ R_{g'} g'=R_{g} g-\Delta\varphi$ and  Gauss-Bonnet theorem $\int R_{g'} g'=8\pi=\int R_{g} g$ to get
\begin{align*}
\int R_{g'} g'(c+X_g)dz=\int R_{g} g(c+X_g)dz-\int \Delta\varphi X_gdz.
\end{align*}
Then  a shift in the Gaussian integral (Girsanov theorem) 
completes the proof.\end{proof}

\noindent The multiplicative factor is called the {\it Weyl anomaly} in physics literature and $c_L$ is the  {\it central charge} of Liouville theory. As a consequence of the Proposition we get M\"obius transformation rule (see see \cite{DKRV})
\begin{corollary}\label{mobius} Let $\psi$ be a M\"obius map of $\hat C$. Then
\begin{align*}
\int F(\phi_{g})d\nu_{g}=\int F(\phi_{g}\circ\psi+Q\ln|\psi'|)d\nu_{g}
\end{align*}
\end{corollary}

\subsection{Vertex operators} Since the  M\"obius group is non-compact the Corollary makes one suspect that the measure $\nu_g$ does not have a  finite mass. Indeed, by Proposition \ref{prop} we may work with the round metric $\hat g$ where $R_{\hat g}=2$. Then  $\frac{1}{4\pi}\int R_{\hat g} \hat gdz=2c$ by Gauss-Bonnet and $\int X_{\hat g}\hat gdz=0$. We get
$$
\int 1 d\nu_{\hat g}=\int\, \E_{\hat g}\, e^{- 2Qc}  e^{-  \mu e^{\gamma c} M_{\hat g,\gamma}{(\C)}}
dc=\infty
$$ 
as the integral diverges at $c\to -\infty$ and $M_{\hat g,\gamma}{(\C)}<\infty$ a.s.. 

Recall that we are looking for a measure with three points on $\hat \C$ fixed. We define (regularized) {vertex operators}
$$
V_{\alpha,\epsilon}(z):=\epsilon^{\frac{_{\alpha^2}}{^2}}e^{\alpha\phi_{\hat g,\epsilon}(z)}
$$
and consider their correlation function 
$$
\langle \prod_{i=1}^nV_{\alpha_i}(z_i)\rangle_{\hat g}:=\lim_{\epsilon\to 0}\int  \prod_{i=1}^nV_{\alpha_i,\epsilon}(z_i)d\nu_{{\hat g}}$$
Now the $c$-integral converges  if and only if  $\sum\alpha_i>2Q$:
$$
\int_\R e^{(\sum\alpha_i-2Q)c-\mu e^{\gamma c}M_{ \hat g,\gamma}(\C)}\,dc = \gamma^{-1}{\mu^{-s}}\Gamma(s) M_{ \hat g,\gamma}(\C)^{-s}
$$
with 
$s=\gamma^{-1}(\sum_i\alpha_i-2Q)$. The remaining expectation over the GFF %
can be dealt with a shift of $X_{\hat g}$ to dispose of $\prod_i e^{\alpha_i X_{{\hat g},\epsilon}(z_i)}$  . The result is after some calculation (\cite{DKRV}, \cite{krv})
\begin{equation*}
\langle \prod_{i=1}^nV_{\alpha_i}(z_i)\rangle_{\hat g}=const. \prod_{j < k} \frac{1}{|z_j-z_k|^{\alpha_j \alpha_k}} \mu^{-s} \gamma^{-1}\Gamma(s)\E\,  M_{ \hat g,\gamma}(F)^{-s}  
\end{equation*}
where
\begin{equation*}
F(z)=   \prod_i \frac{1}{|z-z_i|^{\gamma \alpha_i}}  {\hat  g}(z)^{ - \frac{\gamma}{4} \sum_l \alpha_l } .
\end{equation*}
The {modulus of continuity} of the Chaos measure is
$$
M_{\hat g}(B_r)\leq C(\omega)r^{\gamma Q-\delta}
$$
 for any $\delta>0$. This leads to integrability of $F$ if $\alpha_i<Q$ for all $i$ and
\begin{proposition}
$0<\langle \prod_{i=1}^nV_{\alpha_i}(z_i)\rangle_{\hat g}<\infty$ if and only if {$\sum\alpha_i>2Q$} and  $\alpha_i<Q$.
\end{proposition}
These bounds for $\alpha_i$ are called {\it Seiberg bounds}. Note that they imply that we need at {\it at least three} vertex operators to have a finite correlation function.

\subsection{KPZ conjectures}

Given $z_1,z_2,z_3$ we define the probability measure
$$
d\hat\P_{{\mu,\gamma}}:=
\langle \prod_{i=1}^3V_{\gamma}(z_i)\rangle_{\hat g}^{-1}\prod_{i=1}^nV_{\gamma}(z_i)d\nu_{ \hat g}$$
We may now  state the KPZ conjecture precisely: the random measure $\rho_{\mu,\gamma}$ coming from scaling limit is in law equal to the measure $\caM:=e^{\gamma c}M_{ \hat g,\gamma}$ under $\hat\P_{{\mu,\gamma}}$.
Let $A=\caM(\C)$ be the "volume of the universe". By a simple change of variables in the $c$-integration $e^{\gamma c}M_{g,\gamma}=A$ we obtain
$$
 \E F(A) =\frac{\mu^s}{\Gamma (s)}\int_0^{\infty}F(y)y^se^{-\mu y}\,dy$$
where $s= ({3\gamma-2Q})/\gamma=2-4/\gamma^2$
i.e. under $\P_{\mu,\gamma}$ the law of  $A$ is  $\Gamma(2-4/\gamma^2,\mu ) $.
which agrees with the result in random surfaces. 

The emphasis of KPZ was actually on correlation functions of Conformal Field Theories on random surfaces. As an example, consider the Ising model ($\gamma=\sqrt 3$). 
We can transport the Ising spins $\sigma_v=\pm 1$ sitting at vertices $v$ of $T$ to $\hat\C$. Define the 
 distribution
 \begin{align}\label{ising}
 \Phi_T^{(\epsilon)}(z)=\epsilon^{\frac{5}{6}}\sum_{v\in{\cal V}(T)}\sigma_v\delta(z-\psi_T(v)).
 \end{align}
where $\psi_T:T\to\hat\C$  is the uniformizing map. Then under $\P_{\mu_0+\epsilon\mu,\gamma}$ this becomes a random field on $\hat\C$ and  the KPZ conjecture says that its correlation functions converge (in the sense of distributions) to a product form
 \begin{align*}
\lim_{\epsilon\to 0}\E\Phi^{(\epsilon)}(u_1)\dots \Phi^{(\epsilon)}(u_n)=\E\sigma(u_1)\dots \sigma(u_n)\E_{{\mu}, \gamma}V_\alpha(u_1)\dots V_\alpha(u_n)
\end{align*}
where $\E\sigma(u_1)\dots \sigma(u_n)$ are the correlation functions of the Ising model in the scaling limit on $\hat\C$ and  $\alpha$ is determined from the requirement $\frac{1}{16}+\Delta_{\alpha}=1$ which means that $\sigma(z)e^{\alpha\phi_g(z)}$  transforms under conformal maps as a density.

\subsection{Conformal Field Theory}
So far we have motivated the Liouville model through its conjectural relationship  to scaling limits of random triangulations. However, the Liouville model  is also an interesting {\it Conformal Field Theory} by itself. This way of looking we view the vertex operators as (Euclidean) quantum fields.

First, using the M\"obius invariance (Corollary \ref{mobius})  of $\nu_g$  and taking care with the transformation of the $\epsilon$ in the vertex operator one gets 
$$
\langle \prod_{i=1}^nV_{\alpha_i}(\psi(z_i))\rangle_g
=\prod_i|\psi'(z_i)|^{-2\Delta_{\alpha_i}}\langle \prod_{i=1}^nV_{\alpha_i}(z_i)\rangle_g
$$
where $\Delta_{\alpha}=\frac{\alpha}{2}(Q-\frac{\alpha}{2})$. In CFT parlance, $V_\alpha$ is a {\it primary field} with conformal weight $\Delta_{\alpha}$.

Second, the Liouville model has also {\it local conformal symmetry}. In CFT  this derives from the  {\it energy-momentum tensor}  which encodes the variations of the theory with respect to the background metric. More specifically, one may define the the correlation functions in a smooth Riemannian metric near our  ${g}$ and consider the one parameter family $g^{-1}_\epsilon=g^{-1}+\epsilon f \partial_z\otimes \partial_z
$ where $f$ is a smooth function with support in $\C\setminus \cup_i z_i$. 
 Then
 (a component of) the stress tensor $T(z)$ is defined by the following formula in the physics literature (see \cite{gaw})
 \begin{equation}\label{defgen}  
 \frac{_d}{^{d\epsilon}}\mid_{\epsilon=0} \langle   \prod_l V_{\alpha_l}(z_l)   \rangle_{g_\epsilon}:=     \int f(z)\langle T(z) \prod_l V_{\alpha_l}(z_l)    \rangle_g  g(z)dz.
 \end{equation}
 A simple formal computation then yields the following  heuristic formula 
 \begin{equation}\label{defforus}
 T(z)=  Q \partial_{z}^2 \phi(z)- (( \partial_{z}\phi(z))^2-\E( \partial_{z}X_g(z))^2)
 \end{equation}
 where $\phi$ is the { Liouville field}.
 In the same way, perturbing the metric instead by $\epsilon f \partial_{\bar z}\otimes \partial_{\bar z}$ yields the field $\bar T(z)$.
  
  $T(z)$ encodes { local conformal symmetries} through  the {\it Conformal Ward Identities}. The first Ward identity controls the singularity as the argument of $T$ gets close to one of the $V_\alpha$:
\begin{equation}
  \langle T(z) \prod_l V_{\alpha_l}(z_l)   \rangle_g= \sum_{k} \frac{\Delta_{\alpha_k} }{(z-z_k)^2} \langle  \prod_l V_{\alpha_l}(z_l)   \rangle_g   -\sum_{k} \frac{1}{z-z_k} \partial_{z_k}\langle  \prod_l V_{\alpha_l}(z_l)   \rangle_g  \quad\label{wardid1}
 \end {equation}
 and 
 the second identity controls the singularity when two $T$-insertions come close  \begin{align}  \nonumber
  &  \langle T(z)T(z') \prod_l V_{\alpha_l}(z_l)   \rangle_g
=\frac{\hf c_{\mathrm{L}} }{(z-z')^4}  \langle T(z')T(z) \prod_l V_{\alpha_l}(z_l)   \rangle_g\\&+\frac{2 }{(z-z')^2} \langle T(z') \prod_l V_{\alpha_l}(z_l)   \rangle_g  +\frac{1 }{z-z'} \partial_{z'} \langle T(z') \prod_l V_{\alpha_l}(z_l)   \rangle_g+\dots\label{wardid2}
 \end {align}
where the dots refer to terms that are bounded as $z\to z'$. 
In \cite{krv}  we define $T(z)$ rigorously and prove the Ward identities.

\subsection{Representation Theory} 
Let us finally reconstruct the quantum theory from our probabilistic framework.
Fix the metric $g=\hat g$, the round metric. Let $\caF_\D$ consist of functions $F(\phi)$ measurable w.r.t. the $\sigma$-algebra generated by $\phi |_\D$.
 The measure $\nu_{\hat g}$ is {\it reflection positive}: 
$$(F,G):=\int \overline{ F(X)}(\Theta G)(X)d\nu_{\hat g}(X)\geq 0\ \ \ \forall F,G\in\caF_\D
$$
where 
$(\Theta F)(X):=F(\theta X)$ and $(\theta X)(z)=X(1/\bar z)$.
Define the {Physical Hilbert space} as (here bar denotes completion)
$$\caH:=\overline{\caF_\D/\{F:(F,F)=0\}}
$$
The GFF can be decomposed  to an  independent sum:
$$
X_{\hat g}=X_\D+X_{\D^c}+P\psi
$$
where $X_\D$ and $X_{\D^c}$ are Dirichlet GFF's on $\D$ and  $\D^c$, $\psi$ is the restriction of the GFF to $\partial\D=S^1$ with zero average 
  ("1/f noise") and
$P\psi$ is the harmonic extension of $\psi$ on $\C$.
Let $\E_\D $ be the expectation in the $X_\D$. Then
$$
(UF)(c,\psi):=e^{-Qc}\E_\D (e^{-\mu \int_\D  e^{\gamma \phi}dz}F(\phi))
$$
defines a unitary map 
$$U:\caH\to L^2(\P(d\psi)\,dc)$$
and we may identify $\caH$ with the latter.
The dilation $z\to e^{-t}z$ with $t\geq 0$ acts on $\caF_\D$ and generates a contraction semigroup
$$
e^{-tH}:\caH\to\caH
$$
The generator $H
\geq 0$ is the {Hamiltonian} operator of the CFT. 

  Let $\caV$ be the linear span of the vectors 
  $
  U( \prod_{i=1}^n V_{\alpha_i}(z_i))
  $ with $|z_i|<1$. Then 
$$
L_n=\oint_{|z|=r} z^{n+1}T(z).
$$
acts on $\caV$ by taking $1-r$ small enough. 
The
Ward identities imply the  {\it Virasoro algebra} commutation rules on $\caV$:
$$
[L_m,L_n]=(m-n)L_{m+n}+\frac{_{c_L}}{^{12}}m(m^2-1)\delta_{m,-n}.
$$
The operators satisfy $L_n^\ast=L_{-n}$ on $\caV$. The conjugate field $\bar T$ gives rise to another copy of the Virasoro algebra. A major challenge is to 
study the reduction of this representation  to irreducibles. It is conjectured \cite{rib} that
$\caH $ decomposes to a direct integral $\int_{\R_+}^\oplus \caH_PdP$
where $\caH_P$ is a highest weight module for the two algebras 
with $L_0\psi_P=\Delta_{Q+iP}\psi_P$ and similarly for  $\bar L_0$. 
$\psi_P$ is  formally the  state corresponding to the vertex operator $V_{Q+iP}$ which saturates the Seiberg bound. In \cite{DKRV2} these were constructed for $P=0$. It would be nice to understand the complex case.

\subsection{DOZZ-conjecture} In conformal field theory it is believed \cite{BPZ} that all correlation functions are determined by the knowledge of primary fields (i.e. spectrum of representations) and their three point functions. For the latter  there is a remarkable conjecture due to Dorn, Otto, Zamolodchikov and  Zamolodchikov \cite{Do,ZZ} in Liouville theory. By M\"obius invariance 
$$
\langle V_{\alpha_1}(z_1)V_{\alpha_2} (z_2)V_{\alpha_3}(z_3)\rangle
= |z_1-z_2|^{2 \Delta_{12}} |z_2-z_3|^{2 \Delta_{23}} |z_1-z_3|^{2 \Delta_{13}} C_{\gamma}( \alpha_1, \alpha_2, \alpha_3)
$$
where  $\Delta_{12}= \Delta_{\alpha_3}-\Delta_{\alpha_1}-\Delta_{\alpha_2}  $ etc. and 
$$
C_{\gamma}( \alpha_1, \alpha_2, \alpha_3)=
const.\mu^{-s}\Gamma(s)\,
\E\,  Z^{-s}
$$
with 
$$
Z=\int |z|^{-\alpha_1\gamma}|z-1|^{-\alpha_2\gamma}\hat g(z)^{-\frac{\gamma}{4}\sum_{i=1}^{3}\alpha_i}M_{\hat g,\gamma}(dz)
.$$ 
The   {DOZZ Conjecture} gives an explicit formula for $C_{\gamma}( \alpha_1, \alpha_2, \alpha_3)$. It is based on analyticity and symmetry  assumptions  that lack proofs. One of the ingredients in its derivation was recently proved in \cite{krv} namely  the so-called {\it BPZ equations} \cite{BPZ})  for the vertex operators $V_{\chi}$ with $\chi=({- \frac{\gamma}{2}})^{\pm1}$
(in the language of CFT, these are level two 
degenerate fields). More precisely, 
we prove
\begin{align*}
 (\frac{_1}{^{\chi^2}}
 \partial_{z}^2   + \sum_k (\frac{\Delta_{\alpha_k}}{(z-z_k)^2}   +  
 \frac{1}{z-z_k}  \partial_{z_k}) )\langle   V_{\chi}(z) \prod_i V_{\alpha_i}(z_i)   \rangle     =  0.
\end{align*}
Using the BPZ equation, we recover an explicit  formula found earlier in the physics literature for the 4 point correlation function $\langle    V_{-\frac{\gamma}{2}}  (z)  \prod_{i=1}^3 V_{\alpha_i}(z_i)  \rangle$. Following what is called Teschner's trick \cite{teschner}, we then deduce a non trivial functional relation for $C_{\gamma}( \alpha_1, \alpha_2, \alpha_3)$. The DOZZ formula follows from this relation provided $C_{\gamma}( \alpha_1, \alpha_2, \alpha_3)$ can be extended analytically away from the region $\sum\alpha_i>2Q$ where it is defined. It is a challenge to complete this argument.

\section{References}

\renewcommand{\refname}{}    


\frenchspacing
\begin{thebibliography}{7}



\bibitem{BPZ} A.A. Belavin, A.M.Polyakov,  A.B.Zamolodchikov : Infinite conformal symmetry in two-dimensional quantum field theory, \textit{Nuclear Physics B} \textbf{241} (2)(1984), 333-380 . 

\bibitem{CC}  R. Catellier and K. Chouk: Paracontrolled distributions and the 3-dimensional stochastic quantization equation. ArXiv: 1310.6869 (2013) 

\bibitem{DKRV}
  Liouville Quantum Gravity on the Riemann sphere, \textit{Communications in Mathematical Physics}
, \textbf{342,} (2016),  869-907.

\bibitem{DKRV2}
F.David,  A.Kupiainen,  R.Rhodes,  V.Vargas, Renormalizability of Liouville Quantum Gravity at the Seiberg bound, {http://arxiv.org/abs/1506.01968}.

\bibitem{Do}  H.Dorn,  H.-J.Otto: Two and three-point functions in Liouville theory,   \textit{Nucl.Phys.} \textbf{ B429} (1994) 375-388

\bibitem{DRSVreview}
  B.Duplantier,  R.Rhodes,  S.Sheffield,  V.Vargas, Log-correlated Gaussian fields: an overview, arXiv:1407.5605.
%

\bibitem{gaw}
  K.Gawedzki, Lectures on conformal field theory. In \textit{Quantum fields and strings: A course for mathematicians},  727–805. Amer. Math. Soc., Providence, RI, (1999). 

\bibitem{GIP} M. Gubinelli, P. Imkeller, and N. Perkowski: Paracontrolled distributions and singular PDEs. \textit{Forum Math. Pi.} \textbf{ 3} (2015)

\bibitem{GJ}J. Glimm and A. Jaffe, \textit{Quantum Physics. A Functional Integral Point of View} (Springer Verlag, Berlin, Heidelberg and New York) 1987.

\bibitem{hairerkpz} M. Hairer: Solving the KPZ equation. \textit{Ann. Math.} \textbf{ 178(2)},  559--664 (2013)


\bibitem{hairer} M. Hairer: A theory of regularity structures. \textit{Invent. Math.} \textbf{ 198(2)} (2014), 269-504
\bibitem{hairers} M. Hairer, H. Shen, The dynamical sine-Gordon model	arXiv:1409.5724 [math.PR]
\bibitem{KPZ}  V.G.Knizhnik,  A.M.Polyakov,  A.B.Zamolodchikov,  Fractal structure of 2D-quantum gravity, \textit{Modern Phys. Lett A}, \textbf{3}(8) (1988), 819-826.

\bibitem{AK}
A. Kupiainen, Renormalization group and Stochastic PDEs.
\textit{Ann. Henri Poincaré} \textbf{17(3)} (2016), 497-535 


\bibitem{KM}
A. Kupiainen, M. Marcozzi, Renormalization of Generalized KPZ equation, 	arXiv:1604.08712

\bibitem{krv} A.Kupiainen,  R.Rhodes,  V.Vargas,  Conformal Ward and BPZ Identities for Liouville quantum field theory, 	arXiv:1512.01802 [math.PR]

\bibitem{OS}K. Osterwalder and R. Schrader, Axioms for Euclidean Green’s functions, II, \textit{Commun.
Math. Phys.}, \textbf{42}  (1975), 281–305
\bibitem{Pol}
 A.M.Polyakov, Quantum geometry of bosonic strings, \textit{Phys. Lett. } \textbf{103B} 207 (1981).

\bibitem{rib}
 S.Ribault, Conformal Field theory on the plane, arXiv:1406.4290. 

\bibitem{Simon}B. Simon, The $P(\phi)_2$ - Euclidean (Quantum) Field Theory (Princeton University Press, Princeton) 1974.
\bibitem{Spohn} H. Spohn: Nonlinear Fluctuating Hydrodynamics for Anharmonic Chains. \textit{J. Stat. Phys.} \textbf{ 154} (2014), 1191--1227

\bibitem{SW}R.F. Streater and A.S. Wightman, PCT, Spin and Statistics, and All That (Reading, Mass., Benjamin/Cummings Publ. Co.) 1964.

\bibitem{teschner}
 J.Teschner, On the Liouville three point function, \textit{Phys. Letters} {\textit B363} (1995), 65-70.


\bibitem{wilson} K. Wilson: The renormalization group and critical phenomena. Nobel Lecture. Rev. Mod. Phys. (1984)

\bibitem{ZZ} 
 A.B. Zamolodchikov, Al.B. Zamolodchikov,  Structure constants and conformal bootstrap in
Liouville field theory, \textit{Nucl. Phys. B} \textbf{477}, 577-605 (1996).



\end{thebibliography}
\end{document}